\documentclass[12pt,leqno]{amsart}

\usepackage[margin=1in]{geometry}
\usepackage{amsmath,amssymb,amsthm,mathtools}
\usepackage{enumitem}
\usepackage{microtype}
\usepackage{color}
\usepackage{hyperref}

\hypersetup{colorlinks=true,linkcolor=blue,citecolor=blue,urlcolor=blue}
\numberwithin{equation}{section}

\newtheorem{theorem}{Theorem}[section]
\newtheorem{proposition}[theorem]{Proposition}
\newtheorem{lemma}[theorem]{Lemma}
\newtheorem{corollary}[theorem]{Corollary}
\theoremstyle{definition}
\newtheorem{definition}[theorem]{Definition}
\theoremstyle{remark}

\newcommand{\Ric}{\operatorname{Ric}}
\newcommand{\Rm}{\operatorname{Rm}}
\newcommand{\sgn}{\operatorname{sgn}}
\newcommand{\arsinh}{\operatorname{arsinh}}
\newcommand{\supp}{\operatorname{supp}}
\newcommand{\Vol}{\operatorname{Vol}}

\newcommand{\EA}{\mathcal{E}_A}

\title{An Anisotropic Entropy for the Rotationally Symmetric Ricci Flow}
\author{Hongxin Guo \and Guozhi Lin}

\address{Department of mathematics, Wenzhou University, Wenzhou, Zhejiang 325035, China.}
\email{guo@wzu.edu.cn}
\address{Department of mathematics, Wenzhou University, Wenzhou, Zhejiang 325035, China.}
\email{13610585481@163.com}

\begin{document}

\begin{abstract}
For a rotationally symmetric Ricci flow
$g(t)=ds^2+f(s,t)^2g_{\mathbb S^{n-1}}$, $n\ge3$, we study the
scale-invariant curvature anisotropy
$A=1-f_s^2+ff_{ss}$ introduced by Angenent--Knopf.
At the critical power $p_*=(n-1)/(n-2)$, with
$\Phi=\sgn(A)|A|^{p_*/2}$, we prove the distributional identity
\[
 (\partial_t-\Delta)|A|^{p_*}
 =-\frac4{n-1}
 \left|\Phi_s+(n-1)\frac{f_s}{f}\Phi\right|^2.
\]
Integration against a conjugate heat kernel gives an anisotropic entropy formula. 
As an application, we obtain
an entropy threshold for Type-I singularities on $\mathbb S^n$:
if the entropy associated with a singular-time kernel is at most $1$
at an earlier regular time, every tangent flow at the corresponding singular
point is a round shrinking sphere. This criterion follows from the
$0/1$ quantization of the singular-time anisotropic density.
\end{abstract}


\maketitle
\section{Introduction}\label{sec:introduction}
On a rotationally symmetric manifold with time-dependent Riemannian metric
\begin{equation}
 g(t)=ds^2+f(s,t)^2g_{\mathbb S^{n-1}},\qquad n\ge3,
 \label{eq:intro.metric}
\end{equation}
where $s$ is arclength, the Ricci flow equation is equivalent to
\begin{equation}
	f_t=f_{ss}-(n-2)\frac{1-f_s^2}{f}.
	\label{eq:warp.function}
\end{equation}
The radial and tangential sectional curvatures are given by
\[
 K_{\mathrm{rad}}=-\frac{f_{ss}}{f},\qquad
 K_{\mathrm{tan}}=\frac{1-f_s^2}{f^2}.
\]
The scale-invariant curvature pinching quantity
\begin{equation}
	A=f^2(K_{\mathrm{tan}}-K_{\mathrm{rad}})=1-f_s^2+ff_{ss}
	\label{A definition}
\end{equation}
was introduced by Angenent--Knopf~\cite{AK}.
They derived its evolution and proved its bound by the maximum principle. Their later work gives precise
neckpinch asymptotics~\cite{AK07}. The quantity $A$ vanishes on a round sphere and equals $1$ on a
cylinder. It therefore distinguishes these two singularity models. For $A$ we obtain the following new identity.
\begin{theorem}\label{thm:intro.identity}
Let $g(t)$ be a smooth rotationally symmetric Ricci flow and set
$\displaystyle p_*=(n-1)/(n-2)$ and
$\displaystyle\Phi=\sgn(A)|A|^{p_*/2}$. Then
\begin{equation}
 (\partial_t-\Delta)|A|^{p_*}
 =-\frac4{n-1}
 \left|\Phi_s+(n-1)\frac{f_s}{f}\Phi\right|^2
 \label{eq:intro.identity}
\end{equation}
holds in the sense of distributions.
\end{theorem}
Inspired by Perelman~\cite{Perelman}, we integrate this identity against a conjugate heat kernel to obtain a monotone quantity. 
The same idea was also used by Philipowski, Thalmaier and the first author in~\cite{GPT} to define a new entropy formula.

\begin{theorem}\label{thm:intro.entropy}
Let $g(t)$ be a smooth closed rotational Ricci flow and let $u>0$ solve
$u_t=-\Delta u+Ru$ with $\int_Mu\,d\mu_t=1$. Define
\begin{equation}
 \EA(t;u)=\int_M |A|^{p_*}u\,d\mu_t.
 \label{eq:intro.entropy}
\end{equation}
Then
\begin{equation}
 \frac{d}{dt}\EA(t;u)
 =-\frac4{n-1}\int_M
 \left|\Phi_s+(n-1)\frac{f_s}{f}\Phi\right|^2u\,d\mu_t.
 \label{eq:intro.dissipation}
\end{equation}
\end{theorem}

The integrand in the right hand side of Eq.\eqref{eq:intro.dissipation} 
has a direct geometric interpretation. Writing
\[
 D\Phi:=\Phi_s+(n-1)\frac{f_s}{f}\Phi
       =f^{-(n-1)}\partial_s(f^{n-1}\Phi),
\]
we see that $D\Phi=0$ forces $f^{n-1}\Phi$ to be constant along the
orbit space. A smooth pole makes this constant zero, so $A=0$ and the
sectional curvature is constant. 

Our application  is an entropy criterion excluding cylindrical Type-I
tangents. Suppose that a closed rotationally symmetric Ricci flow on
$\mathbb S^n$ satisfies
$\sup_{\mathbb S^n}|\Rm(g(t))|\le C/(T-t)$ with $T$ the singular time, and let $p$ be a Type-I
singular point. Let $\mathcal U_p$ be the family of singular-time
conjugate heat kernels based at $(p,T)$, as in
Mantegazza--M\"uller~\cite{MM}. For $u\in\mathcal U_p$, Theorem~\ref{thm:intro.entropy}
gives the limit
\begin{equation}
 \Theta_A(p;u)=\lim_{t\uparrow T}\mathcal E_A(t;u).
 \label{eq:intro.density}
\end{equation}
The dissipation formula determines the possible values of this density.

\begin{theorem}[Singular-time density and tangent selection]\label{thm:intro.density}
Let $g(t)$, $0\le t<T$, be a smooth closed rotationally symmetric
Ricci flow on $\mathbb S^n$, $n\ge3$, satisfying the Type-I curvature
bound, and let $p$ be a Type-I singular point. The limit
$\Theta_A(p;u)$ is independent of $u\in\mathcal U_p$. Denoting it by
$\Theta_A(p)$ which is called the singular-time density, one has $\Theta_A(p)\in\{0,1\}$. Every tangent flow at
$p$ is a round shrinking sphere when $\Theta_A(p)=0$, and a round
shrinking cylinder when $\Theta_A(p)=1$.
\end{theorem}

\begin{samepage}
\begin{corollary}[Entropy threshold for spherical tangents]\label{cor:intro.threshold}
Under the assumptions of Theorem~\ref{thm:intro.density}, if for some
$u\in\mathcal U_p$ and some regular time $t_0<T$,
\begin{equation}
 \mathcal E_A(t_0;u)\le1,
 \label{eq:intro.threshold}
\end{equation}
then every tangent flow at $p$ is a round shrinking sphere.
In particular, if $\|A(\cdot,t_0)\|_{L^\infty}\le1$, every Type-I
singular point has spherical tangents. A cylindrical tangent
at $p$ requires $\mathcal E_A(t;u)>1$ for every $u\in\mathcal U_p$
and every $t<T$.
\end{corollary}
\end{samepage}

The weighted condition~\eqref{eq:intro.threshold} is weaker than a
pointwise bound on $|A|$: it allows larger anisotropy in regions carrying
less mass for the chosen singular-time kernel.

\begingroup
\emergencystretch=2em
Let us summarize the relevant known results. Enders--M\"uller--Topping showed that Type-I blow-ups at singular points
admit nonflat shrinking solitons as subsequential limits~\cite{EMT}.
Kotschwar classified complete rotational shrinkers~\cite{Kotschwar}.
Related approaches to singularities through entropy and local curvature
control appear in~\cite{MM,BDM22,BDM26,Yokota,Hallgren}.
Colding--Minicozzi proved cylindrical tangent uniqueness without
rotational symmetry~\cite{CM25}, and Fang--Li obtained fixed-gauge
convergence~\cite{FangLi}.
\par\endgroup

Our emphasis is the critical square identity,
its exact entropy dissipation, and the resulting threshold criterion.
The proof of Theorem~\ref{thm:intro.density} uses zero-dissipation
rigidity to select the tangent model without invoking the rotational
shrinker classification.

Section~\ref{sec:identity} proves Theorem~\ref{thm:intro.identity},
including the weak extension across zeros and smooth poles.
Section~\ref{sec:entropy} proves Theorem~\ref{thm:intro.entropy} and
develops the pointed representation and zero-dissipation rigidity.
Section~\ref{sec:application} applies these formulas to density
quantization and the entropy threshold. The compactness arguments for
rotational tangent geometry are collected in Appendix~\ref{app:geometry};
kernel convergence, entropy preservation, and the weak limit of the
dissipation identity are proved in Appendix~\ref{app:kernels}.
\subsection*{Acknowledgements} The authors thank ChatGPT for assistance with some of the mathematical discussions and for language polishing during the preparation of this paper.
 All mathematical content, derivations, and conclusions presented herein were carried out and verified by the authors, who take full responsibility for the mathematical content of the paper.

\section{The critical square identity}\label{sec:identity}
\subsection{Rotational evolution and the critical power}
Time derivatives below keep the orbit label $x$ fixed; $\partial_s$ denotes arclength differentiation in the evolving metric. Then
\begin{equation}
 [\partial_t,\partial_s]=-(n-1)\frac{f_{ss}}{f}\partial_s.
 \label{eq:2.1}
\end{equation}
Under a constant parabolic rescaling
$\widetilde g(\tau)=\lambda g(t_0+\tau/\lambda)$, one has
\begin{equation}
 \widetilde s=\sqrt\lambda\,s,\qquad
 \widetilde f=\sqrt\lambda\,f,\qquad
 \widetilde A=A.
 \label{eq:2.4}
\end{equation}
The last equality follows directly from~\eqref{A definition}.

The evolution equation of $A$~\cite[Lemma~3.1]{AK}, written in our $n$-dimensional notation, is
\begin{equation}
 A_t=A_{ss}+(n-5)\frac{f_s}{f}A_s-4(n-2)\frac{f_s^2}{f^2}A.
 \label{eq:2.5}
\end{equation}
For a radial function $v=v(s)$,
\begin{equation}
 \Delta v=v_{ss}+(n-1)\frac{f_s}{f}v_s.
 \label{eq:2.6}
\end{equation}
Hence
\begin{equation}
 (\partial_t-\Delta)A=-4\frac{f_s}{f}A_s-4(n-2)\frac{f_s^2}{f^2}A.
 \label{eq:2.7}
\end{equation}

We now seek the power of $|A|$ whose evolution has no zeroth-order remainder.

\begin{proposition}[Power evolution]\label{prop:3.1}
For every $p>1$, at points where $A\neq 0$,
\begin{align}
 (\partial_t-\Delta)|A|^p
 &=-p(p-1)|A|^{p-2}
 \left(A_s+\frac{2}{p-1}\frac{f_s}{f}A\right)^2 \notag\\
 &\quad -4p\left(n-2-\frac{1}{p-1}\right)\frac{f_s^2}{f^2}|A|^p.
 \label{eq:3.1}
\end{align}
\end{proposition}

\begin{proof}
The chain rule and~\eqref{eq:2.7} give
\begin{align*}
(\partial_t-\Delta)|A|^p
&=-p(p-1)|A|^{p-2}A_s^2
  -4p\frac{f_s}{f}|A|^{p-2}AA_s\\
&\quad-4p(n-2)\frac{f_s^2}{f^2}|A|^p.
\end{align*}
Completing the square yields~\eqref{eq:3.1}.
\end{proof}

The remainder in~\eqref{eq:3.1} vanishes only at
\begin{equation}
 p_*=\frac{n-1}{n-2}.
 \label{eq:3.2}
\end{equation}
Set
\begin{equation}
 \Phi=\sgn(A)|A|^{p_*/2}.
 \label{eq:3.3}
\end{equation}

\begin{theorem}[Critical square identity]\label{thm:3.2}
On regular orbits away from the zero set of $A$,
\begin{equation}
 (\partial_t-\Delta)|A|^{p_*}
 =-\frac{4}{n-1}
 \left|\Phi_s+(n-1)\frac{f_s}{f}\Phi\right|^2.
 \label{eq:3.4}
\end{equation}
Equivalently,
\begin{equation}
 (\partial_t-\Delta)|A|^{p_*}
 =-\frac{4}{n-1}f^{-2(n-1)}
 \left|\partial_s(f^{n-1}\Phi)\right|^2.
 \label{eq:3.5}
\end{equation}
\end{theorem}

\begin{proof}
At $p=p_*$ the remainder in~\eqref{eq:3.1} vanishes, while
\[
 \Phi_s=\frac{p_*}{2}|A|^{p_*/2-1}A_s,\qquad
 \frac{2}{p_*-1}=\frac{2(n-1)}{p_*},\qquad
 \frac{4(p_*-1)}{p_*}=\frac4{n-1}.
\]
These identities give~\eqref{eq:3.4}; the product rule gives~\eqref{eq:3.5}.
\end{proof}

\subsection{Pole behavior and weak extension}
We establish the regularity at a smooth pole that is required to extend the square identity across it.
At a smooth pole, geodesic normal coordinates with $e_1\perp e_2$ give $f(s,t)^2=s^2g_{22}(se_1,t)$, where rotational symmetry makes $g_{22}(re_1,t)$ smooth and even in signed $r$, with value $1$ at $r=0$. Thus
\begin{equation}
 f=s+a(t)s^3+b(t)s^5+O(s^7).
 \label{eq:4.1}
\end{equation}
The expansion and its derivatives hold uniformly on compact time intervals.

\begin{lemma}[Fourth-order vanishing]\label{lem:4.1}
Near a smooth pole,
\begin{equation}
 A=(10b-3a^2)s^4+O(s^6),\qquad A_s=O(s^3).
 \label{eq:4.2}
\end{equation}
\end{lemma}

\begin{proof}
Expand
\[
 f_s=1+3as^2+5bs^4+O(s^6),\qquad
 f_{ss}=6as+20bs^3+O(s^5),
\]
and substitute into $A=1-f_s^2+ff_{ss}$. The quadratic terms cancel and~\eqref{eq:4.2} follows.
\end{proof}

\begin{lemma}[Weak form across the zero set]\label{lem:3.3}
Let $p_*=(n-1)/(n-2)>1$ and $\Phi=\sgn(A)|A|^{p_*/2}$. Then $\Phi\in W^{1,2}_{\mathrm{loc}}$ in the orbit variable and~\eqref{eq:3.4} holds in the sense of distributions across $\{A=0\}$. In particular,
\begin{equation}
 D\Phi:=\Phi_s+(n-1)\frac{f_s}{f}\Phi\in L^2_{\mathrm{loc}},
 \label{eq:3.6}
\end{equation}
and the right-hand side of~\eqref{eq:3.4} is the locally integrable function
$-\frac{4}{n-1}|D\Phi|^2$.
\end{lemma}

\begin{proof}
For $p_*=2$, the chain rule and~\eqref{eq:3.7} are classical; the pole check below still applies. For $p=p_*\in(1,2)$, take $0\le\eta\in C_c^1(J)$ on a compact orbit interval and set
\[
 G_\varepsilon(z)=z(z^2+\varepsilon)^{(p-2)/2}.
\]
Since $G_\varepsilon'\ge(p-1)(z^2+\varepsilon)^{(p-2)/2}$ and $|G_\varepsilon(z)|\le |z|^{p-1}$, integration by parts gives
\[
 (p-1)\int_J \eta(A^2+\varepsilon)^{(p-2)/2}A_s^2\,ds
 \le-\int_J(\eta_sA_s+\eta A_{ss})G_\varepsilon(A)\,ds\le C_\eta.
\]
On compact time intervals the derivatives of $A$ are bounded, so $C_\eta$
can be chosen independently of time. Monotone convergence in the orbit
variable and then integration over a compact time interval yield
\begin{equation}
 |A|^{p-2}A_s^2\in L^1_{\mathrm{loc}}(J).
 \label{eq:3.7}
\end{equation}

Put $F_\varepsilon(z)=(z^2+\varepsilon)^{p/2}$. Its parabolic chain rule is
\[
 (\partial_t-\Delta)F_\varepsilon(A)
 =F_\varepsilon'(A)(\partial_t-\Delta)A
 -F_\varepsilon''(A)A_s^2,
\]
where $|F_\varepsilon'(A)|\le p|A|^{p-1}$ and
$0\le F_\varepsilon''(A)A_s^2\le p|A|^{p-2}A_s^2$ almost everywhere; the exceptional zeros with $A_s\ne0$ have measure zero. By~\eqref{eq:3.7}, dominated convergence gives, distributionally,
\[
 (\partial_t-\Delta)|A|^p
 =p|A|^{p-2}A(\partial_t-\Delta)A
 -p(p-1)|A|^{p-2}A_s^2.
\]
Completing the square as in Proposition~\ref{prop:3.1} proves~\eqref{eq:3.4}.
The same estimate applies in a signed pole coordinate: the smooth even
normal-coordinate extension described before Lemma~\ref{lem:4.1} makes
$A$ smooth there. Lemma~\ref{lem:4.1} also gives
$q_s=O(s^{4p_*-1})$ for $q=|A|^{p_*}$, so $f^{n-1}q_s\to0$ and no
pole distribution arises.

For $p\in(1,2)$, let
$\Phi_\varepsilon=A(A^2+\varepsilon)^{(p-2)/4}$.
Then $\Phi_\varepsilon\to\Phi$ pointwise and
$|\partial_s\Phi_\varepsilon|^2\le |A|^{p-2}A_s^2$
almost everywhere. By~\eqref{eq:3.7}, dominated convergence gives
$\Phi_\varepsilon\to\Phi$ in $W^{1,2}_{\mathrm{loc}}$ on compact time intervals.
For either value of $p_*$, almost everywhere,
\[
 \partial_s\Phi=\frac{p_*}{2}|A|^{p_*/2-1}A_s,
\]
so~\eqref{eq:3.7} gives $\Phi_s\in L^2_{\mathrm{loc}}$. At a pole, Lemma~\ref{lem:4.1} gives
\[
 \frac{f_s}{f}\Phi=O\bigl(s^{2p_*-1}\bigr),
\]
which is square-integrable. Hence $D\Phi\in L^2_{\mathrm{loc}}$ and the right-hand side of~\eqref{eq:3.4} is locally integrable.
\end{proof}

\section{The anisotropic entropy and equality case}\label{sec:entropy}
\subsection{Entropy dissipation and pointwise consequences}
Let $u>0$ solve the conjugate heat equation
\begin{equation}
 u_t=-\Delta u+Ru,\qquad \int_M u\,d\mu=1.
 \label{eq:5.1}
\end{equation}
Since $\partial_t d\mu=-R\,d\mu$, integration by parts gives
\begin{equation}
 \frac{d}{dt}\int_M \phi u\,d\mu
 =\int_M (\partial_t-\Delta)\phi\,u\,d\mu.
 \label{eq:5.2}
\end{equation}

\begin{definition}\label{def:3.2}
Define the anisotropic entropy
\begin{equation}
 \mathcal{E}_A(t;u)=\int_M |A|^{p_*}u\,d\mu.
 \label{eq:5.3}
\end{equation}
\end{definition}

\begin{theorem}\label{thm:5.2}
Let $g(t)$ be a smooth closed rotational Ricci flow on a time interval, and let $u>0$ solve~\eqref{eq:5.1} there. Then
\begin{equation}
 \frac{d}{dt}\mathcal{E}_A(t;u)
 =-\frac{4}{n-1}\int_M
 \left|\Phi_s+(n-1)\frac{f_s}{f}\Phi\right|^2u\,d\mu.
 \label{eq:5.4}
\end{equation}
In particular, $\mathcal{E}_A$ is nonincreasing.
\end{theorem}

\begin{proof}
Put $q=|A|^{p_*}$. Since $p_*>1$, $q$ is $C^1$ even where $A=0$. Using~\eqref{eq:5.1} and $\partial_t d\mu=-R\,d\mu$, we obtain
\begin{align*}
 \frac{d}{dt}\mathcal{E}_A(t;u)
 &=\int_M(q_tu+qu_t-Rqu)\,d\mu\\
 &=\int_M(q_tu-q\Delta u)\,d\mu\\
 &=\int_M(q_t-\Delta q)u\,d\mu\\
 &=-\frac{4}{n-1}\int_M|D\Phi|^2u\,d\mu.
\end{align*}
The last two equalities hold weakly across $\{A=0\}$ by Lemma~\ref{lem:3.3}. At a smooth pole, Lemma~\ref{lem:4.1} gives $q=O(r^{4p_*})$, $q_r=O(r^{4p_*-1})$ and $f=O(r)$, where $r$ is the distance to the pole. Since $u$ and its first derivatives are bounded, the boundary flux satisfies
\[
 f^{n-1}(q\,\partial_r u-uq_r)
 =O(r^{n+4p_*-2})\longrightarrow0.\qedhere
\]
\end{proof}

\noindent\textbf{Pointed entropy representation.}
For $x\in M$ and $0\le t_0<t_1<T$, let
$H_{(x,t_1)}(y,t)$ be the conjugate heat kernel based at $(x,t_1)$.
Apply Theorem~\ref{thm:5.2} on $[t_0,t_1-\varepsilon]$ and let
$\varepsilon\downarrow0$. Using
$H_{(x,t_1)}(\cdot,t)\,d\mu_t\rightharpoonup\delta_x$ as
$t\uparrow t_1$ gives the terminal value because $|A|^{p_*}$ is
continuous at the regular time $t_1$. The nonnegative dissipation
integral converges by monotone convergence. Hence
\begin{align}
 |A(x,t_1)|^{p_*}
 &+\frac4{n-1}\int_{t_0}^{t_1}\int_M
 |D\Phi|^2H_{(x,t_1)}\,d\mu_t\,dt\notag\\
 &=\int_M |A(y,t_0)|^{p_*}H_{(x,t_1)}(y,t_0)
 \,d\mu_{t_0}(y).
 \label{eq:5.pointed}
\end{align}

\begin{proposition}[Uniform anisotropy bound]\label{prop:7.1}
Let $g(t)$, $0\le t<T$, be a smooth closed rotationally symmetric Ricci flow. Then
\begin{equation}
 \|A(\cdot,t)\|_{L^\infty(M,g(t))}
 \le \|A(\cdot,0)\|_{L^\infty(M,g(0))},
 \qquad 0\le t<T.
 \label{eq:7.1}
\end{equation}
The same bound holds on every constant parabolic rescaling of the flow.
If the flow has a smooth pole and $A(\cdot,t_0)\not\equiv0$, then
the bound is strict at every later regular time:
\begin{equation}
 \|A(\cdot,t_1)\|_{L^\infty}<
 \|A(\cdot,t_0)\|_{L^\infty},
 \qquad 0\le t_0<t_1<T.
 \label{eq:7.strict}
\end{equation}
\end{proposition}

\begin{proof}
Formula~\eqref{eq:5.pointed}, with $t_0=0$, gives
$|A(x,t)|^{p_*}\le\|A(\cdot,0)\|_\infty^{p_*}$ for every $x$ and
$t<T$. Taking the $p_*$th root proves~\eqref{eq:7.1}. At either smooth pole,
Lemma~\ref{lem:4.1} gives $A=O(s^4)$, so the pointwise kernel limit
also holds there. Finally, $A$ is unchanged by constant parabolic
rescaling by~\eqref{eq:2.4}.
This recovers the bound of~\cite[Corollary~3.2]{AK} through the
pointed entropy identity.

For the strict statement, put
$M_0=\|A(\cdot,t_0)\|_\infty^{p_*}>0$. At a smooth pole,
$|A(\cdot,t_0)|^{p_*}=0$, so it is strictly below $M_0$ on an open
set. The terminal kernel $H_{(x,t_1)}(\cdot,t_0)$ is positive
everywhere. Its average of $|A(\cdot,t_0)|^{p_*}$ is therefore
strictly less than $M_0$ for every $x$. Formula~\eqref{eq:5.pointed}
and compactness of $M$ make the inequality uniform in $x$, proving
\eqref{eq:7.strict}.
\end{proof}

\subsection{Zero-dissipation rigidity}
\begin{theorem}[Rotational zero-dissipation rigidity]\label{thm:6.2}
On any smooth connected rotational time slice with a smooth pole, the weak
condition $D\Phi=0$ implies $A\equiv0$ and constant sectional curvature.

More generally, let
\[
 g(t)=ds^2+f(s,t)^2g_{\mathbb{S}^{n-1}}
\]
be a smooth connected complete rotational Ricci flow on a time interval, and assume that on every time slice $D\Phi=0$ in the weak orbit-variable sense of Lemma~\ref{lem:3.3}. Then exactly one of the following occurs:
\begin{enumerate}[label=(\roman*)]
 \item $A\equiv0$, and every time slice is a space form;
 \item $f_s\equiv0$, $A\equiv1$, and the flow is a round product
 \[
 (B^1,ds^2)\times\bigl(\mathbb{S}^{n-1},f(t)^2g_{\mathbb{S}^{n-1}}\bigr),
 \]
 where $B^1$ is a complete connected flat one-manifold (hence $\mathbb{R}$ or a circle). In particular, when the orbit space is noncompact this is the standard round product cylinder.
\end{enumerate}
If a smooth pole is present, only the first alternative occurs.
\end{theorem}

\begin{proof}
From $D\Phi=0$,
\begin{equation}
 f^{n-1}\Phi=C(t).
 \label{eq:6.1}
\end{equation}
At a smooth pole, $f=0$ and $\Phi$ is bounded by Lemma~\ref{lem:4.1},
so $C=0$ even on a single time slice. Thus $A=0$; the rotational sectional
curvatures agree and Schur's lemma gives constant sectional curvature.

For the flow assertion, at any time when $C(t)=0$ one has $A=0$.
On a time interval where $C(t)\neq0$, the identity
$p_*=(n-1)/(n-2)$ turns~\eqref{eq:6.1} into
\begin{equation}
 A=c(t)f^{-2(n-2)}
 \label{eq:6.2}
\end{equation}
for a nonzero function $c(t)$.

Put $m=n-2$. Differentiating~\eqref{eq:6.2} at fixed orbit
label and using~\eqref{eq:2.5} gives
\[
 \frac{A_t}{A}=\frac{c'}c-2m\frac{f_t}{f}
 =2m\left(-\frac{f_{ss}}f+n\frac{f_s^2}{f^2}\right).
\]
The warping equation~\eqref{eq:warp.function} therefore cancels the
$f_{ss}$ terms and yields
\begin{equation}
 \frac{c'}c
 =\frac{2m}{f^2}\bigl(2(n-1)f_s^2-m\bigr).
 \label{eq:6.3}
\end{equation}
The left side is independent of $s$. Thus, at each time with
$c(t)\neq0$, there is a number $\kappa(t)$ such that
\[
 f_s^2=\frac{m}{2(n-1)}+\kappa(t)f^2.
\]
If $f$ were nonconstant, then on some open orbit interval $f_s\neq0$,
and differentiating in $s$ would give $f_{ss}=\kappa f$. Hence
$A=1-f_s^2+ff_{ss}=n/[2(n-1)]$ would be constant on that interval,
contradicting~\eqref{eq:6.2} with $c\neq0$ and nonconstant $f$.
Therefore $f_s=0$ and $A=1$ whenever $C(t)\neq0$.
Continuity in time prevents a switch between $A=0$ and $A=1$.
In the product case,~\eqref{eq:warp.function} becomes
\[
 f_t=-\frac{n-2}{f},
\]
and the flat complete base is $\mathbb{R}$ or a circle.
\end{proof}

\section{Application: an entropy threshold for Type-I singularities}
\label{sec:application}
The formulas of Sections~\ref{sec:identity} and~\ref{sec:entropy} have
two consequences for Type-I tangents. Their dissipation vanishes in
the limit, so the equality case selects the model; their entropy
then gives a threshold excluding cylindrical tangents. We present
this argument here and defer its compactness and kernel estimates
to the appendices.

\subsection{Singular-time kernels and entropy defect}\label{subsec:density}
Throughout this section and the appendices, assume the closed rotational flow on $\mathbb{S}^n$ satisfies Hamilton's Type-I bound~\cite{Hamilton95Singularities},
\begin{equation}
 |\Rm|_{g(t)}\le\frac{C_I}{T-t}.
 \label{eq:8.1}
\end{equation}
Fix a Type-I singular point $p$ in the sense of
\cite[Definition~1.2]{EMT} and any sequence $\lambda_i\to\infty$.
Set
\begin{equation}
 g_i(\tau)=\lambda_i g\left(T+\frac{\tau}{\lambda_i}\right),\qquad \tau<0.
 \label{eq:8.2}
\end{equation}
By~\cite[Theorem~1.4]{EMT}, every such sequence has a pointed smooth subsequence converging to a nonflat gradient shrinking Ricci soliton in Hamilton's compactness framework~\cite{Hamilton95Compactness}.

For $p\in M$, let $\mathcal{U}_p$ be the family of singular-time Type-I adjoint heat kernels based at $(p,T)$ in the sense of Mantegazza--M\"uller~\cite{MM}. Each $u\in\mathcal{U}_p$ is a positive smooth solution of
\begin{equation}
 u_t=-\Delta u+Ru
 \label{eq:9.1}
\end{equation}
on $M\times[0,T)$, obtained as a locally smooth limit of adjoint heat kernels based at $(p,s_j)$ with $s_j\uparrow T$, and
\begin{equation}
 \int_M u(\cdot,t)\,d\mu_t=1.
 \label{eq:9.2}
\end{equation}
\begin{definition}[Pre-density]\label{def:9.1}
For $u\in\mathcal{U}_p$, specialize the entropy of Definition~\ref{def:3.2} to the $p$-based pre-density
\begin{equation}
 \mathcal{E}_{A,p,u}(t):=\mathcal E_A(t;u)
 \label{eq:9.3}
\end{equation}
and define its singular-time density by
\begin{equation}
 \Theta_A(p;u)=\lim_{t\uparrow T}\mathcal{E}_{A,p,u}(t).
 \label{eq:9.4}
\end{equation}
\end{definition}

By Theorem~\ref{thm:5.2}, normalization~\eqref{eq:9.2} and
Proposition~\ref{prop:7.1}, this limit exists and satisfies
$0\le\Theta_A(p;u)\le\|A(\cdot,0)\|_\infty^{p_*}$.

Integrating~\eqref{eq:5.4} to a regular
time $s<T$ and then letting $s\uparrow T$ gives, for every
$u\in\mathcal U_p$ and $t<T$,
\begin{equation}
 \mathcal{E}_{A,p,u}(t)-\Theta_A(p;u)
 =\frac4{n-1}\int_t^T\int_M |D\Phi|^2u\,d\mu_\sigma\,d\sigma.
 \label{eq:9.defect}
\end{equation}
For a blow-up sequence~\eqref{eq:8.2}, define
\begin{equation}
 u_i(x,\tau)=\lambda_i^{-n/2}
 u\left(x,T+\frac{\tau}{\lambda_i}\right).
 \label{eq:9.6}
\end{equation}
Then $\int u_i\,d\mu_{g_i}=1$, and the scale invariance of $A$ gives
\begin{equation}
 \int_M |A_i|^{p_*}u_i\,d\mu_{g_i(\tau)}
 =\mathcal{E}_{A,p,u}\left(T+\frac{\tau}{\lambda_i}\right).
 \label{eq:9.7}
\end{equation}

\subsection{Density quantization and tangent selection}\label{subsec:selection}
We use three facts proved in the appendices: rotational structure and
local smooth convergence of $A_i$ (Proposition~\ref{prop:8.4});
convergence of the kernels without loss of mass or anisotropic entropy
(Theorem~\ref{thm:9.5}); and zero dissipation of the limiting flow
(Proposition~\ref{prop:10.1}). These facts allow the equality case of
Theorem~\ref{thm:intro.entropy} to determine the singularity model.

\begin{theorem}[Quantization and tangent-type selection]\label{thm:10.2}
Let $g(t)$, $0\le t<T$, be a smooth closed rotationally symmetric Ricci flow on $\mathbb{S}^n$, $n\ge3$, satisfying the Type-I bound~\eqref{eq:8.1}, and let $p$ be a Type-I singular point. The limit $\Theta_A(p;u)$ has the same value for every $u\in\mathcal{U}_p$; denoting it by $\Theta_A(p)$, one has
\begin{equation}
 \Theta_A(p)\in\{0,1\}.
 \label{eq:10.3}
\end{equation}
Every tangent flow at $p$ is a round shrinking sphere when $\Theta_A(p)=0$, and a round shrinking product cylinder when $\Theta_A(p)=1$.
\end{theorem}

\begin{proof}
Fix $u\in\mathcal U_p$ and a tangent sequence. The compactness
results in Appendices~\ref{app:geometry} and~\ref{app:kernels} give a
complete rotational limit and a positive limiting conjugate heat kernel.
For $\tau_1<\tau_2<0$, the rescaled dissipation identity gives
\[
 \frac4{n-1}\int_{\tau_1}^{\tau_2}\int_M
 |D_i\Phi_i|^2u_i\,d\mu_{g_i(\tau)}\,d\tau
 \le \mathcal E_{A,p,u}\left(T+\frac{\tau_1}{\lambda_i}\right)
       -\Theta_A(p;u)\longrightarrow0.
\]
Positivity of the limiting kernel therefore forces the tangent flow to
have zero dissipation. Proposition~\ref{prop:10.1} justifies this passage
through the zero set of $A_\infty$ and at smooth poles, and gives
$D_\infty\Phi_\infty=0$ on every time slice.

By~\cite{EMT}, the tangent is a complete nonflat shrinking soliton.
The zero-dissipation rigidity of Theorem~\ref{thm:6.2} now gives exactly
two possibilities: a space form with $A_\infty=0$, or a round product
with $A_\infty=1$.

In the space-form case, the soliton argument in the proof of Lemma~\ref{lem:8.1} excludes negative curvature, and nonflatness excludes zero; hence the limit is compact and positively curved. A convergence map from this compact limit embeds it into $\mathbb{S}^n$ with image both open and closed, so it is the round sphere, not a quotient.

In the product case, the base is $\mathbb{R}$ or a circle. The circle is impossible: restrict
\[
 \Ric+\nabla^2\psi=\frac12g
\]
to the circle direction. Since $\Ric_{ss}=0$ for the product, integrating $\psi_{ss}=1/2$ around the circle is impossible. Thus the base is $\mathbb{R}$.

For any $u\in\mathcal{U}_p$ and any tangent flow, Theorem~\ref{thm:9.5} and~\eqref{eq:9.12} give
\[
 \Theta_A(p;u)=\int_{M_\infty}|A_\infty|^{p_*}u_\infty\,d\mu_\infty
 =\begin{cases}
 0,&\text{for a round sphere},\\
 1,&\text{for a round cylinder}.
 \end{cases}
\]
For one fixed $u$, this value fixes the type of every tangent; applying the same formula to any other $u$ proves kernel independence.
\end{proof}

This proves Theorem~\ref{thm:intro.density}.

\subsection{The entropy threshold}
The density calculation turns the exact dissipation formula into a
criterion at an earlier regular time.

\begin{corollary}[An entropy threshold for cylindrical tangents]
\label{cor:10.initial}
Under the assumptions of Theorem~\ref{thm:10.2}, if a Type-I singular
point $p$ has a cylindrical tangent, then for every $u\in\mathcal U_p$
and $t<T$,
\begin{equation}
 \mathcal E_{A,p,u}(t)>1,\qquad
 \|A(\cdot,t)\|_{L^\infty}>1.
 \label{eq:10.threshold}
\end{equation}
Consequently, if $\mathcal E_{A,p,u}(t_0)\le1$ for some
$u\in\mathcal U_p$ and some regular time $t_0<T$, every tangent flow
at $p$ is a round shrinking sphere. If
$\|A(\cdot,t_0)\|_{L^\infty}\le1$, the same conclusion holds at every
Type-I singular point.
\end{corollary}

\begin{proof}
For a cylindrical tangent, Theorem~\ref{thm:10.2} gives
$\Theta_A(p;u)=1$ for every $u\in\mathcal U_p$. Hence
\begin{equation}
 \mathcal E_{A,p,u}(t)-1
 =\frac4{n-1}\int_t^T\int_M |D\Phi|^2u\,d\mu_\sigma\,d\sigma
 \ge0.
 \label{eq:threshold.defect}
\end{equation}
If equality held, positivity of $u$ would give $D\Phi=0$ on almost
every smooth time slice in $(t,T)$. On such a slice,
$f^{n-1}\Phi$ is constant along the orbit space, and a smooth pole
forces this constant to vanish. Thus $A\equiv0$ and the entropy is
zero on that slice. Monotonicity and nonnegativity then imply
$\Theta_A(p;u)=0$, a contradiction. This proves the strict entropy
inequality. Since $u$ has unit mass,
\[
 \mathcal E_{A,p,u}(t)\le\|A(\cdot,t)\|_\infty^{p_*},
\]
which proves the second inequality in~\eqref{eq:10.threshold}.
The spherical conclusions follow by contraposition and the two
alternatives of Theorem~\ref{thm:10.2}. This also proves
Corollary~\ref{cor:intro.threshold}.
\end{proof}

The constant $1$ is the anisotropy entropy of the round product
cylinder. Formula~\eqref{eq:threshold.defect} identifies the excess
above this value with the entire remaining weighted dissipation.
The weighted hypothesis concerns a singular-time kernel based at the
chosen point; it does not require $|A|\le1$ everywhere. The criterion
is sufficient for spherical tangents: an entropy value above $1$
alone does not assert that the tangent is cylindrical.

\subsection{Full blow-up convergence}
Once the density selects a tangent type, subsequential compactness
gives the following consequence.
\begin{corollary}[Subsequence-free blow-up convergence]\label{cor:10.3}
Under the assumptions of Theorem~\ref{thm:10.2}, the full family of rescalings
\[
 g_\lambda(\tau)=\lambda g\left(T+\frac{\tau}{\lambda}\right),\qquad \lambda\to\infty,
\]
converges, without passing to subsequences and in the pointed smooth Cheeger--Gromov topology on compact spacetime subsets of $(-\infty,0)$, to the normalized round shrinking sphere if $\Theta_A(p)=0$, and to the normalized round shrinking cylinder if $\Theta_A(p)=1$.
\end{corollary}

\begin{proof}
Every rescaling sequence has a convergent subsequence by~\cite{EMT}. Theorem~\ref{thm:10.2} fixes its type, and the normalized models are
\[
 (-2(n-1)\tau)g_{\mathbb{S}^n}
 \quad\text{and}\quad
 dz^2+(-2(n-2)\tau)g_{\mathbb{S}^{n-1}},
\]
with unit sectional curvature on each sphere. Both are homogeneous, so each type has one pointed normalized model up to spacetime isometry. Failure of full convergence would give a sequence staying away from that model yet admitting a subsequence converging to it, a contradiction.
\end{proof}

\appendix
\section{Rotational compactness of Type-I tangent flows}\label{app:geometry}
\label{subsec:tangent}
We work under the assumptions of Section~\ref{sec:application}, with
the rescaled flows~\eqref{eq:8.2}. The purpose of this appendix is to
prove that their pointed limits inherit the global rotational
structure, with local smooth convergence of the anisotropy $A_i$.
The uniform bound needed below is Proposition~\ref{prop:7.1}.

\subsection{Orbit-radius control}
The rotational orbit through $p$ has radius
\begin{equation}
 F_i(p,\tau)=\sqrt{\lambda_i}\,f\left(p,T+\frac{\tau}{\lambda_i}\right)
 \label{eq:8.3}
\end{equation}
(zero when $p$ is a pole).

\begin{lemma}[Orbit-radius nonescape]\label{lem:8.1}
For every compact time interval $I\Subset(-\infty,0)$, discard finitely many indices so that all $g_i$ are defined on $I$. The rescaled orbit radius at the base point is then uniformly bounded:
\begin{equation}
 \sup_i\sup_{\tau\in I}F_i(p,\tau)<\infty.
 \label{eq:8.4}
\end{equation}
If $p$ is a pole, the assertion is trivial since $F_i(p,\tau)=0$.
\end{lemma}

\begin{proof}
If~\eqref{eq:8.4} failed, take $\tau_i\to\tau_0\in I$ with $F_i(p,\tau_i)\to\infty$. Type-I metric distortion gives
\[
 C^{-1}g_i(\tau_0)\le g_i(\tau_i)\le Cg_i(\tau_0).
\]
Thus $F_i(p,\tau_0)\to\infty$.

Use signed radial distance $r$ from the base orbit. The tangential curvature formula gives
\[
 |F_i'|\le\sqrt{1+CF_i^2}.
\]
Thus $\arsinh(\sqrt{C}F_i)$ is uniformly Lipschitz, so $F_i(0)\to\infty$ implies
\begin{equation}
 F_i(r)\to\infty
 \label{eq:8.5}
\end{equation}
uniformly on compact radial intervals. The distance $d_i$ from the
base orbit to either pole satisfies
\[
 d_i\ge\frac{\arsinh(\sqrt C F_i(0))}{\sqrt C}\longrightarrow\infty,
\]
so these intervals exist for large $i$.

Every fixed pointed ball projects into a bounded radial interval, so
\eqref{eq:8.5} holds uniformly there. The uniform bound for $A_i$
from Proposition~\ref{prop:7.1} now gives
\[
 K_{\mathrm{tan},i}-K_{\mathrm{rad},i}=\frac{A_i}{F_i^2}\longrightarrow0,
 \qquad K_{\mathrm{tan},i}=\frac{1-(F_i')^2}{F_i^2}
 \le\frac1{F_i^2}\longrightarrow0
\]
on these balls. Hence the smooth pointed limit is pointwise isotropic
with nonpositive sectional curvature. By Schur's lemma ($n\ge3$), its
sectional curvature is a constant $K\le0$ on the connected time slice.

If $K=0$, the tangent is flat, contrary to its nonflatness~\cite{EMT}.
If $K<0$, the time slice $g_0=g_\infty(\tau_0)$ has curvature $K$, and its shrinking-soliton equation gives
\[
 \nabla^2\psi=cg_0,
 \qquad c=\frac{1}{2|\tau_0|}-(n-1)K>0.
\]
With $V=\nabla\psi$, commuting $\nabla_XV=cX$ gives $R(X,Y)V=0$. But constant negative curvature then yields
\[
 0=\Ric_{g_0}(V,V)=(n-1)K|V|_{g_0}^2.
\]
Thus $V=0$ and $\nabla^2\psi=0$, contradicting $c>0$. Hence orbit-radius escape is impossible.
\end{proof}

\subsection{Regular and pole charts}
\begin{lemma}[Radial compactness away from poles]\label{lem:8.2}
Let $I\Subset(-\infty,0)$ and let $Q_i$ be a sequence of rotational orbits in the rescaled flows such that their radii satisfy
\[
 0<c\le F_i(Q_i,\tau)\le C
\]
for $\tau\in I$. Then, after passing to a subsequence, there are signed-distance coordinates $r$ from $Q_i$ on a uniform spacetime neighborhood for which
\[
 g_i(\tau)=dr^2+F_i(r,\tau)^2g_{\mathbb{S}^{n-1}},
\]
and the functions $F_i$ converge in $C^\infty$ on compact subcylinders to a positive smooth limit $F_\infty$.
\end{lemma}

\begin{proof}
On a slightly larger time interval, Type-I control and Shi's
estimates~\cite{Shi} give
\[
 |\nabla^j\Rm(g_i)|\le C_j\qquad(\tau\in I).
\]
Put $H_i=K_{\mathrm{rad},i}$, and bound sectional curvatures by
$\Lambda$. At $r=0$,
\[
 c\le F_i\le C,\qquad
 |(F_i)_r|^2=1-K_{\mathrm{tan},i}F_i^2\le1+\Lambda C^2,
 \qquad (F_i)_{rr}=-H_iF_i.
\]
These bounds put $Q_i$ a uniform positive distance from either pole
and give a uniform $\delta>0$ on which
$c/2\le F_i\le C'$ and $|(F_i)_r|\le C'$.
In a radially parallel frame, Shi's estimates also bound $\partial_r^jH_i$. Differentiating $(F_i)_{rr}=-H_iF_i$ gives
\[
 \partial_r^{m+2}F_i
 =-\sum_{j=0}^{m}\binom{m}{j}
 (\partial_r^jH_i)(\partial_r^{m-j}F_i).
\]
Thus all radial derivatives are uniformly bounded.

Let $v_i=\partial_\tau r_i|_x$ be the velocity of signed arclength from the fixed-label orbit $Q_i$. Equations~\eqref{eq:2.1} and~\eqref{eq:warp.function} give
\begin{align*}
 v_i(r,\tau)&=(n-1)\int_0^r\frac{(F_i)_{\rho\rho}}{F_i}\,d\rho,\\
 \left.\partial_\tau F_i\right|_r
 &=(F_i)_{rr}-(n-2)\frac{1-(F_i)_r^2}{F_i}-v_i(F_i)_r.
\end{align*}
Differentiating both equations in $r$ and $\tau$ and inducting on the time order gives
\[
 \sup_{|r|\le\delta,\,\tau\in I}
 |\partial_r^m\partial_\tau^bF_i|\le C_{m,b},
 \qquad m,b\ge0.
\]
Arzel\`a--Ascoli yields smooth subsequential convergence with $F_\infty>0$.
\end{proof}

\begin{lemma}[Pole recentering compactness]\label{lem:8.3}
Let $I\Subset(-\infty,0)$ and let $Q_i$ be a sequence of rotational orbits with $F_i(Q_i,\tau_0)\to0$ at some $\tau_0\in I$. Then after passing to a subsequence each $Q_i$ lies at distance tending to zero from one fixed pole. In radial distance $\rho$ from that pole,
\[
 g_i(\tau)=d\rho^2+\widehat F_i(\rho,\tau)^2g_{\mathbb{S}^{n-1}},
\]
and $\widehat F_i$ subconverges smoothly on compact pole-centered spacetime neighborhoods to a smooth warping function satisfying
\[
 \widehat F_\infty(0,\tau)=0,\qquad
 (\widehat F_\infty)_\rho(0,\tau)=1.
\]
\end{lemma}

\begin{proof}
On a slightly larger time interval, bound sectional curvatures by $K$.
Type-I metric distortion on $I$ gives
$\varepsilon_i:=\sup_{\tau\in I}F_i(Q_i,\tau)
\le C(I,C_I)F_i(Q_i,\tau_0)\to0$, while
\begin{equation}
 |1-(F_i')^2|\le K F_i^2.
 \label{eq:8.8}
\end{equation}
For large $i$, $|F_i'|\ge1/2$ where $F_i\le\varepsilon_i$; its sign is fixed over $I$, and the decreasing direction reaches a pole $P_i$ (take $P_i=Q_i$ if already a pole), with
\[
 \sup_{\tau\in I}d_{g_i(\tau)}(Q_i,P_i)\le2\varepsilon_i\longrightarrow0.
\]
After a subsequence, choose the same pole for all $i$.

Set $\rho=d_{g_i(\tau)}(P_i,\cdot)$ and $H_i=K_{\mathrm{rad},i}$.
Let $L_i(\tau)$ be the distance to the opposite pole. On
$0\le\rho<L_i(\tau)$, the Jacobi equation and the pole conditions
$\widehat F_i(0)=0$, $(\widehat F_i)_\rho(0)=1$ give
\[
 \begin{aligned}
 \widehat F_i(\rho,\tau)
 &=\rho-\int_0^\rho(\rho-\sigma)H_i(\sigma,\tau)
   \widehat F_i(\sigma,\tau)\,d\sigma,\\
 |\widehat F_i(\rho,\tau)-\rho|&\le C_K\rho^3,
 \qquad |(\widehat F_i)_\rho(\rho,\tau)-1|\le C_K\rho^2
 \quad(\rho\le\delta).
 \end{aligned}
\]
Here the estimates follow from $|H_i|\le K$ and the integral equation
on $[0,\min\{L_i(\tau),\delta\})$. Choose $\delta>0$ depending only
on $K$ so that $C_K\delta^2<1/2$. Then
$\widehat F_i\ge\rho/2$ there; if $L_i(\tau)<\delta$, continuity at
the opposite pole would give $0=\widehat F_i(L_i(\tau),\tau)
\ge L_i(\tau)/2$, a contradiction. Thus $L_i(\tau)\ge\delta$
uniformly, and the displayed estimates hold on $[0,\delta]$.
Shi's bounds and
$\partial_\rho^{m+2}\widehat F_i=-\partial_\rho^m(H_i\widehat F_i)$
control every radial derivative up to the pole.

For the pole-centered arclength velocity $v_i=\partial_\tau\rho|_x$,
equation~\eqref{eq:2.1} gives the nonsingular formula
\[
 v_i=-(n-1)\int_0^\rho H_i(\sigma,\tau)\,d\sigma.
\]
The curvature evolution equation and Shi's estimates bound temporal
and spatial derivatives of $H_i$ at fixed orbit labels. In the moving
pole gauge,
\[
 \left.\partial_\tau H_i\right|_\rho
 =\left.\partial_\tau H_i\right|_x-v_i\partial_\rho H_i.
\]
The formula for $v_i$ and induction therefore bound all mixed
derivatives of $H_i$ on $[0,\delta]\times I$. Differentiating the
Jacobi integral equation, whose initial data at the pole are
independent of time, gives the corresponding bounds for
$\widehat F_i$. Its smooth odd pole extension preserves these bounds,
so
\[
 \sup_{[0,\delta]\times I}|\partial_\rho^m\partial_\tau^b\widehat F_i|
 \le C_{m,b}\qquad(m,b\ge0).
\]
Arzel\`a--Ascoli gives smooth convergence up to the pole; Lemma~\ref{lem:8.2} covers adjoining annuli.
\end{proof}

\subsection{Global rotational structure}
\begin{proposition}[Rotational structure of tangent flows]\label{prop:8.4}
Every tangent flow at a Type-I singular point is a complete connected cohomogeneity-one rotational Ricci flow. After passing to compatible radial gauges on compact spacetime subsets, its time slices can be written as
\begin{equation}
 g_\infty(\tau)=dr^2+F_\infty(r,\tau)^2g_{\mathbb{S}^{n-1}},
 \label{eq:8.9}
\end{equation}
and
\begin{equation}
 A_i\longrightarrow A_\infty
 \label{eq:8.10}
\end{equation}
locally smoothly in those gauges.
\end{proposition}

\begin{proof}
Fix $I\Subset(-\infty,0)$, and let $C_0$ bound $F_i(p,\tau)$ on $I$ by Lemma~\ref{lem:8.1}. Let $r$ be orbit-space arclength from the orbit of $p$, signed when that orbit is regular and one-sided when $p$ is a pole. Choose $K\ge1$ bounding the absolute sectional curvature on $I$. The tangential curvature formula gives $|(F_i)_r|\le\sqrt{1+K F_i^2}$, hence, for $|r|\le R$,
\[
 \arsinh(\sqrt K F_i(r,\tau))
 \le\arsinh(\sqrt K C_0)+\sqrt K R,
 \qquad F_i(r,\tau)\le C_R.
\]
For $q$ on the orbit at $r$,
\[
 |r|\le d_{g_i(\tau)}(p,q)\le |r|+\pi F_i(p,\tau)
 \le |r|+\pi C_0.
\]
Thus bounded pointed balls correspond to bounded radial intervals,
including whole orbits.

Let $\Psi_i$ be pointed smooth convergence embeddings supplied by
\cite{EMT}. For a point $q$ of the limit at time $\tau_0$, let $Q_i$
be the orbit through $\Psi_i(q)$. After a subsequence,
$F_i(Q_i,\tau_0)\to b\ge0$; metric distortion preserves this
alternative over $I$. Lemma~\ref{lem:8.2} gives a regular chart if
$b>0$, and Lemma~\ref{lem:8.3} a pole chart if $b=0$. In the latter
case,~\eqref{eq:8.8} gives
\[
 d_{g_i(\tau)}(Q_i,\{P_i^-,P_i^+\})\le2F_i(Q_i,\tau).
\]
Thus zero-radius points are poles, isolated by
$\widehat F_\infty(\rho,\tau)=\rho+O(\rho^3)$. For each integer
$R>0$ and each compact time interval, choose finitely many limit
points whose regular or pole neighborhoods cover the closed pointed
$R$-ball. Apply the two compactness lemmas to their representing orbits
and to every overlap on a common subsequence. Because these
neighborhoods lie in a larger pointed ball, their limits identify with
open subsets of the limit defined by $\Psi_i$. Diagonal extraction over
$R$ and the time intervals gives compatible charts. Choose their
origins according to the original ordered orbit labels, and use one global angular
trivialization of $\mathbb S^n$ minus its two poles throughout this
diagonal construction. An orbit meeting $B_{g_i(\tau)}(p,R)$ has
$|r|\le R$, and the preceding full-orbit bound places its entire
spherical fiber in $B_{g_i(\tau)}(p,R+\pi C_0)$. Thus the limit sees
each full round $\mathbb S^{n-1}$ fiber, without an angular quotient.
On regular overlaps,
\[
 (r_\beta,\theta_\beta)
 =\bigl(\pm r_\alpha+c_{\alpha\beta}(\tau),
 Q_{\alpha\beta}\theta_\alpha\bigr),\qquad
 F_\beta(r_\beta,\tau)=F_\alpha(r_\alpha,\tau),
\]
where $Q_{\alpha\beta}\in O(n)$ is fixed by the original angular
trivialization. The inherited order makes the limiting orbit space an
interval. A finite endpoint with positive $F_\infty$ extends by
Lemma~\ref{lem:8.2}; any finite endpoint is therefore a smooth pole by
Lemma~\ref{lem:8.3}. The charts glue to a global rotational warped-product
structure on the pointed limit, with one time-independent rotational
action. In each radial chart,
\[
 A_i=1-(F_i)_r^2+F_i(F_i)_{rr}
 \longrightarrow 1-(F_\infty)_r^2+F_\infty(F_\infty)_{rr}=A_\infty
\]
smoothly.
\end{proof}

\section{Kernel convergence and vanishing dissipation}\label{app:kernels}
We retain the assumptions and notation of Section~\ref{sec:application}.
Fix $u\in\mathcal U_p$ and rescale it by~\eqref{eq:9.6} along a tangent
sequence. Appendix~\ref{app:geometry} supplies the rotational structure
and smooth convergence of $A_i$ used below.

\subsection{Kernel compactness and tail control}
\begin{lemma}[Compactness of the rescaled kernels]\label{lem:9.3}
After passing to a subsequence along any pointed Type-I blow-up, the functions $u_i$ converge locally smoothly on $M_\infty\times(-\infty,0)$ to a positive solution $u_\infty$ of the conjugate heat equation of the limit flow.
\end{lemma}

\begin{proof}
The rescaled Gaussian bounds of~\cite[Propositions~2.7--2.8]{MM}
control $u_i$ above and below on compact pointed sets. Shi's
estimates~\cite{Shi} and interior parabolic regularity give locally
smooth subsequential convergence; the lower bound gives positivity.
\end{proof}

\begin{lemma}[Uniform tail control]\label{lem:9.4}
Fix a compact interval $J\Subset(-\infty,0)$ and discard finitely many indices so that all rescaled flows are defined on $J$. Then, for every $\tau\in J$,
\begin{equation}
 \lim_{R\to\infty}\sup_i
 \int_{M\setminus B_{g_i(\tau)}(p,R)}u_i\,d\mu_{g_i(\tau)}=0,
 \label{eq:9.8}
\end{equation}
and
\begin{equation}
 \lim_{R\to\infty}\sup_i
 \int_{M\setminus B_{g_i(\tau)}(p,R)}|A_i|^{p_*}u_i\,d\mu_{g_i(\tau)}=0.
 \label{eq:9.9}
\end{equation}
The convergence as $R\to\infty$ is uniform for $\tau\in J$.
\end{lemma}

\begin{proof}
The Gaussian upper bound of Mantegazza--M\"uller~\cite[Proposition 2.8]{MM} rescales, uniformly for $i$ and $\tau\in J$, to
\begin{equation}
 u_i(x,\tau)\le C\exp\bigl(-c\,d_{g_i(\tau)}(p,x)^2\bigr),
 \qquad \tau\in J.
 \label{eq:9.10}
\end{equation}
The Type-I bound gives $\Ric_{g_i(\tau)}\ge-Kg_i(\tau)$; Bishop comparison yields, uniformly in $i$ and $\tau\in J$,
\begin{equation}
 \Vol_{g_i(\tau)}B(p,R)\le C_0e^{C_0R},\qquad R\ge1,
 \label{eq:9.11}
\end{equation}
Summing over the annuli $B(p,k+1)\setminus B(p,k)$ gives
\[
 \sup_{i,\tau\in J}
 \int_{M\setminus B_{g_i(\tau)}(p,R)}u_i\,d\mu_{g_i(\tau)}
 \le C\sum_{k\ge\lfloor R\rfloor}e^{-ck^2+C_0(k+1)}\longrightarrow0.
\]
This is~\eqref{eq:9.8}. Since $|A_i|\le C_A$ by Proposition~\ref{prop:7.1} and scale invariance, the weighted tail is at most $C_A^{p_*}$ times the mass tail, proving~\eqref{eq:9.9}.
\end{proof}

\subsection{Preservation of mass and anisotropic entropy}
\begin{theorem}[No loss of critical anisotropy entropy]\label{thm:9.5}
Let a subsequence of $(M,g_i(\tau),p)$ converge smoothly to a tangent flow $(M_\infty,g_\infty(\tau),p_\infty)$, and choose the kernel subsequence supplied by Lemma~\ref{lem:9.3}. Then
\begin{equation}
 \int_{M_\infty}u_\infty\,d\mu_{g_\infty(\tau)}=1
 \label{eq:9.12}
\end{equation}
for every $\tau<0$, and
\begin{equation}
 \Theta_A(p;u)
 =\int_{M_\infty}|A_\infty|^{p_*}u_\infty\,d\mu_{g_\infty(\tau)}.
 \label{eq:9.13}
\end{equation}
In particular, the right-hand side is independent of $\tau$.
\end{theorem}

\begin{proof}
Fix $\tau<0$. Let $\Psi_i:U_i\to M$ be pointed convergence maps, and choose $\chi_R\in C_c^\infty(M_\infty)$ with
\[
 0\le\chi_R\le1,\qquad
 \chi_R\equiv1\ \text{on }B_\infty(p_\infty,R),\qquad
 \supp\chi_R\subset B_\infty(p_\infty,R+1).
\]
For fixed $R$ and large $i$, $\supp\chi_R\subset U_i$.
For $h_i=u_i$ or $h_i=|A_i|^{p_*}u_i$, let $h_\infty$ be the corresponding limit. Smooth convergence of $A_i,u_i$ and continuity of $z\mapsto|z|^{p_*}$ give local uniform convergence of $h_i$. Thus, with the pullback cutoff extended by zero,
\[
 \int_{M_\infty}\chi_Rh_\infty\,d\mu_\infty
 =\lim_{i\to\infty}\int_M(\chi_R\circ\Psi_i^{-1})h_i\,d\mu_i.
\]
For fixed $R$ and large $i$, pointed smooth convergence gives
\[
 B_{g_i(\tau)}(p,R-1)\subset\Psi_i(B_\infty(p_\infty,R))
 \subset B_{g_i(\tau)}(p,R+1).
\]
Hence, for either choice of $h_i$,
\[
 0\le\int_Mh_i\,d\mu_i
 -\int_M(\chi_R\circ\Psi_i^{-1})h_i\,d\mu_i
 \le\int_{M\setminus B_{g_i(\tau)}(p,R-1)}
 h_i\,d\mu_i\le\eta(R),
\]
where $\eta(R)\to0$ by Lemma~\ref{lem:9.4}. Consequently,
\[
 0\le\lim_{i\to\infty}\int_Mh_i\,d\mu_i
       -\int_{M_\infty}\chi_Rh_\infty\,d\mu_\infty\le\eta(R).
\]
Letting $R\to\infty$ therefore gives
\[
 \int_{M_\infty}u_\infty\,d\mu_\infty=1,
 \qquad
 \int_{M_\infty}|A_\infty|^{p_*}u_\infty\,d\mu_\infty
 =\lim_{i\to\infty}\int_M|A_i|^{p_*}u_i\,d\mu_i.
\]
By~\eqref{eq:9.7}, the last limit is $\Theta_A(p;u)$. This proves~\eqref{eq:9.12}--\eqref{eq:9.13} for each $\tau<0$.
\end{proof}

\subsection{The weak limit of the dissipation identity}
The dissipation formula is invariant under Type-I rescaling: if $D_i$ is~\eqref{eq:3.6} computed in $g_i$, then
\[
 D_i\Phi_i=\lambda_i^{-1/2}D\Phi,
\]
and the spacetime integral of $|D\Phi|^2u$ is unchanged.

\begin{proposition}[Caloricity from vanishing entropy drop]\label{prop:10.1}
For every tangent flow obtained at $p$, the bounded function
$q_\infty:=|A_\infty|^{p_*}$ is caloric:
\begin{equation}
 (\partial_\tau-\Delta_{g_\infty(\tau)})q_\infty=0
 \quad\text{on }M_\infty\times(-\infty,0).
 \label{eq:10.1}
\end{equation}
Moreover $D_\infty\Phi_\infty=0$ in the weak orbit-variable sense on
every time slice, including across the zero set of $A_\infty$.
\end{proposition}

\begin{proof}
Fix $\tau_1<\tau_2<0$. The defect identity~\eqref{eq:9.defect}
is invariant under the parabolic rescaling and gives the full
remaining dissipation tail
\begin{align}
 &\mathcal{E}_{A,p,u}\left(T+\frac{\tau_1}{\lambda_i}\right)
 -\Theta_A(p;u)\notag\\
 &\qquad=\frac4{n-1}\int_{\tau_1}^{0}\int_M
 |D_i\Phi_i|^2u_i\,d\mu_{g_i(\tau)}\,d\tau
 \longrightarrow0.
 \label{eq:10.2}
\end{align}
On any compact
spacetime set $K\Subset M_\infty\times(\tau_1,\tau_2)$ in the pointed
limit, the rescaled kernels satisfy
$u_i\ge c_K>0$ for large $i$ by Lemma~\ref{lem:9.3}. Pulling back
by the convergence maps and writing the pullbacks with the same
symbols, we obtain
\begin{equation}
 \int_K|D_i\Phi_i|^2\,d\mu_{g_i(\tau)}\,d\tau
 \le \frac{n-1}{4c_K}
 \left[\mathcal E_{A,p,u}\!\left(T+\frac{\tau_1}{\lambda_i}\right)
       -\Theta_A(p;u)\right]\longrightarrow0.
 \label{eq:10.localdefect}
\end{equation}

The weak square identity~\eqref{eq:3.4} holds for each rescaled flow.
The locally smooth convergence $A_i\to A_\infty$ implies
$|A_i|^{p_*}\to q_\infty$ locally uniformly. Passing to distributions
in the square identity therefore gives~\eqref{eq:10.1}; parabolic
regularity makes $q_\infty$ smooth. Applying the same weak identity to
the smooth limit flow then gives
$|D_\infty\Phi_\infty|^2=0$ almost everywhere in spacetime.
On $\{A_\infty\neq0\}$, $D_\infty\Phi_\infty$ is continuous and
therefore vanishes at every time. On each fixed time slice,
$\Phi_\infty\in W^{1,2}_{\mathrm{loc}}$ by Lemma~\ref{lem:3.3};
its weak orbit derivative vanishes almost everywhere on the level set
$\{\Phi_\infty=0\}$. The lower-order term in $D_\infty\Phi_\infty$
also vanishes there. Hence $D_\infty\Phi_\infty=0$ weakly on every
time slice, including at smooth poles.
\end{proof}

\end{document}